\documentclass[11pt,reqno]{amsart}

\usepackage{amsmath, amsthm, amsopn, amssymb}
\usepackage{enumerate}
\usepackage{hyperref}

\newtheorem{thm}{Theorem}[section]
\newtheorem{lemma}[thm]{Lemma}
\newtheorem{cor}[thm]{Corollary}
\newtheorem{prop}[thm]{Proposition}

\theoremstyle{definition}
\newtheorem*{claims}{Claim}

\newcommand{\eps}{\varepsilon}

\newcommand{\ex}{\mathrm{ex}}
\newcommand{\exc}{\mathrm{excess}}
\newcommand{\len}{\lambda}
\newcommand{\maxdev}{\mathrm{maxdev}}

\newcommand{\cA}{\mathcal{A}}
\newcommand{\cB}{\mathcal{B}}
\newcommand{\HH}{\mathcal{H}}
\newcommand{\NN}{\mathbb{N}}
\newcommand{\cP}{\mathcal{P}}

\newcommand{\Mb}{M^*}

\newcommand{\Bin}{\mathrm{Bin}}
\newcommand{\Ex}{\mathbb{E}}

\renewcommand{\le}{\leqslant}
\renewcommand{\ge}{\geqslant}

\newcommand{\ple}{\preccurlyeq}
\newcommand{\pl}{\prec}

\title{Long paths and cycles in random subgraphs of $\HH$-free graphs}

\date{\today}

\author{Michael Krivelevich}
\address{School of Mathematical Sciences, Tel Aviv University, Tel Aviv 69978, Israel}
\email{krivelev@post.tau.ac.il}

\author{Wojciech Samotij}
\address{School of Mathematical Sciences, Tel Aviv University, Tel Aviv 69978, Israel; and Trinity College, Cambridge CB2 1TQ, UK}
\email{samotij@post.tau.ac.il}

\thanks{Research supported in part by:
  (MK) USA-Israel BSF Grant 2010115 and by grant 912/12 from the Israel Science Foundation;
  (WS) ERC Advanced Grant DMMCA and a grant from the Israel Science Foundation.
}

\begin{document}

\maketitle

\begin{abstract}
  Let $\HH$ be a given finite (possibly empty) family of connected graphs, each containing a~cycle, and let $G$ be an arbitrary finite $\HH$-free graph with minimum degree at least $k$. For $p \in [0,1]$, we form a $p$-random subgraph $G_p$ of $G$ by independently keeping each edge of $G$ with probability~$p$. Extending a classical result of Ajtai, Koml{\'o}s, and Szemer{\'e}di, we prove that for every positive~$\eps$, there exists a positive $\delta$ (depending only on $\eps$) such that the following holds: If $p \ge \frac{1+\eps}{k}$, then with probability tending to $1$ as $k \to \infty$, the random graph $G_p$ contains a cycle of length at least $n_\HH(\delta k)$, where $n_\HH(k)>k$ is the minimum number of vertices in an $\HH$-free graph of average degree at least~$k$. Thus in particular $G_p$ as above typically contains a cycle of length at least linear in $k$.
\end{abstract}

\section{Introduction}
\label{sec:introduction}

Given a graph $G$ and a real number $p \in [0,1]$, we define the \emph{$p$-random subgraph of $G$}, denoted by $G_p$, to be the random subgraph of $G$ such that each edge of $G$ belongs to $G_p$ with probability $p$, independently of all other edges. The most studied case of the above model is when $G$ is a complete graph. This particular model, usually denoted by $G(n,p)$, where $n$ is the number of vertices in the base (complete) graph, was first introduced in~\cite{Gi59} and has since become one of the most popular objects of study in combinatorics.

In the groundbreaking paper of Erd{\H{o}}s and R{\'e}nyi~\cite{ErRe60}, the following fundamental discovery was made: If we let $p(n)$ gradually increase from $0$ to $1$, then the connectivity structure of the random graph $G(n,p)$ undergoes a dramatic \emph{phase transition} around $p(n) = \frac{1}{n}$. For any positive constant~$\eps$, if $p(n) \le \frac{1-\eps}{n}$, then asymptotically almost surely\footnote{That is, with probability tending to $1$ as $n \to \infty$.} (a.a.s.), the size of each connected component of $G(n,p)$ is at most logarithmic in $n$, whereas if $p(n) \ge \frac{1+\eps}{n}$, then a.a.s.\ $G(n,p)$ has a unique component of linear size, traditionally called the \emph{giant component}. The paper of Erd{\H{o}}s and R{\'e}nyi has had an enormous influence on the development of the theory of random graphs. Its main results have been given several different proofs and extended or improved in many different ways. For a detailed account of the theory of random graphs, we refer the reader to the two standard monographs~\cite{Bo01, JaLuRu00}.

One of the better-known extensions of the main result of~\cite{ErRe60} is due to Ajtai, Koml{\'o}s, and Szemer{\'e}di~\cite{AjKoSz81}, who proved that if $p(n) \ge \frac{1+\eps}{n}$, then not only the random graph $G(n,p)$ a.a.s.\ contains a giant component occupying a positive proportion of all the vertices, but also it typically has a path of length linear in $n$. An easy corollary of this fact is that if $p(n) \ge \frac{1+\eps}{n}$, then a.a.s.\ $G(n,p)$ contains a cycle of length linear in $n$. A simple proof of this result has recently been given in~\cite{KrSu}.

The following natural generalizations of the classical results about the evolution of the random graph mentioned above were recently considered in~\cite{FrKr, KrLeSu, KrSu, Ri}. Suppose that $p_0 \colon \NN \to [0,1]$ is a \emph{threshold function} for some (monotone) graph property $\cP$ in $G(n,p)$. For example, one might let $p_0(n) = \frac{1}{n}$ and $\cP$ be the property of containing a connected component (or a path / cycle) of size (length) linear in $n$. In particular, suppose that if $p \ge (1+\eps)p_0$ for some positive constant~$\eps$, then a.a.s.\ $G(n,p)$ possesses $\cP$. Does this statement remain true if one replaces $G(n,p)$ with the $p$-random subgraph $G_p$ of an \emph{arbitrary} graph $G$ with minimum degree $n-1$? This has been answered in the affirmative in several cases, e.g., when $\cP$ is the property of being non-planar~\cite{FrKr}, containing a path of length linear in $n$~\cite{KrSu} or of length at least $n-1$~\cite{KrLeSu}, and having a cycle of length
$n-o(n)$~\cite{KrLeSu}. As a by-product, more robust proofs of the corresponding statements in the case $G = K_n$ were obtained.

In this paper, we continue this study and introduce one additional twist. Namely, we fix a finite family $\HH$ of graphs and further assume that the base graph $G$ is \emph{$\HH$-free}, i.e., that $G$ does not contain a copy of any $H \in \HH$ as a subgraph. Since we allow the family $\HH$ to be empty, which imposes no additional restrictions on $G$, our results will generalize some previous works. Our aim is to prove that if $G$ is an $\HH$-free graph with minimum degree at least $k$ (from now on, we will use~$k$ to denote the lower bound on the minimum degree and $n$ to denote the number of vertices of the graph $G$) and $p \ge \frac{1+\eps}{k}$, then with probability approaching $1$ as $k \to \infty$, the random graph $G_p$ contains a long path and a long cycle. Since we are only interested in the asymptotic behavior of these probabilities, we will assume that there exist $\HH$-free graphs with arbitrary large minimum degree. This implies, in particular, that the family $\HH$ cannot contain any acyclic graphs. Moreover, we will assume for convenience that every graph in $\HH$ is connected. We will term such families \emph{good}. That is, we will say that a family $\HH$ of graphs is good if it is a finite collection of connected graphs, each containing a cycle. Finally, we will assume throughout the paper that the base graph~$G$ is finite.

Recall that the \emph{Tur{\'a}n number for $\HH$}, denoted by $\ex(n,\HH)$, is the maximum number of edges in an $\HH$-free graph on $n$ vertices. Observe that if $G$ is an $\HH$-free graph with average degree at least~$k$, then the number $n$ of vertices of $G$ satisfies
\begin{equation}
  \label{eq:n-lower}
  n k \le 2 \cdot \ex(n,\HH).
\end{equation}
Let $n_\HH(k)$ be the smallest $n$ for which~\eqref{eq:n-lower} holds, that is, the smallest number of vertices in an $\HH$-free graph of average degree at least $k$. Our main result is the following.

\begin{thm}
  \label{thm:cycle}
  For every positive $\eps$, there exists a positive constant $\delta$ such that the following is true. Let $k$ be a sufficiently large integer, let $\HH$ be a good family of graphs, and let $G$ be an $\HH$-free graph with minimum degree at least $k$. If $p \ge \frac{1+\eps}{k}$, then
  \[
  \Pr\big(\text{$G_p$ contains a cycle of length at least $n_\HH(\delta k)$}\big) \ge 1 - \exp( - \delta k ).
  \]
\end{thm}

It is plausible that $n_\HH(\delta k)$ above could be replaced by $\delta n_\HH(k)$, but unfortunately our methods do not yield this stronger statement, cf.~Section~\ref{sec:dep-ell-eps}. Note that $n_\HH(k) \ge k+1$ for every family $\HH$ and that $n_\HH(k) \le 2k$ if $\HH$ does not contain any bipartite graph. Consequently, the interesting cases will be either when $\HH$ is empty or when it contains at least one bipartite graph. Indeed, when we let $\HH$ be the empty family, we obtain the following corollary.

\begin{cor}
  \label{cor:cycle}
  For every positive $\eps$, there exists a positive constant $\delta$ such that the following is true. Let $k$ be a sufficiently large integer and let $G$ be an arbitrary graph with minimum degree at least $k$. If $p \ge \frac{1+\eps}{k}$, then
  \[
  \Pr\big(\text{$G_p$ contains a cycle of length at least $\delta k$}\big) \ge 1 - \exp( - \delta k ).
  \]
\end{cor}

We remark that the statement obtained from Corollary~\ref{cor:cycle} by replacing `a cycle of length $\delta k$' with `a path of length $\eps^2k/5$' was proved in~\cite{KrSu}. Still, when $G$ is an arbitrary graph, one cannot easily deduce Corollary~\ref{cor:cycle} from its `path version' using a standard double exposure (sprinkling) argument as, unlike the case when $G$ is a complete graph, there is no guarantee that $G$ contains any edges that close a given path of length $\Theta(k)$ into a cycle of comparable length. For example, the girth of $G$ can be much larger than $k$.

Another case which seems especially interesting is when $\HH$ is the family of all cycles of lengths ranging from $3$ to some $g$. Note that in this case, requiring a graph to be $\HH$-free is the same as requiring that its girth exceeds $g$. Since if $g$ is even, then $\ex(n,C_g) = O(n^{1+2/g})$, as proved by Bondy and Simonovits~\cite{BoSi74}, Theorem~\ref{thm:cycle} has the following nice corollary.

\begin{cor}
  \label{cor:cycle-girth}
  For every $g \in \{2, 3, \ldots\}$ and every positive $\eps$, there exists a positive constant $\delta$ such that the following is true. Let $k$ be a sufficiently large integer and let $G$ be an arbitrary graph with minimum degree at least $k$ and girth larger than $g$. If $p \ge \frac{1+\eps}{k}$, then
  \[
  \Pr\big(\text{$G_p$ contains a cycle of length at least $\delta k^{\lfloor g / 2 \rfloor}$}\big) \ge 1 - \exp( - \delta k ).
  \]
\end{cor}

Observe that Corollary~\ref{cor:cycle-girth} remains true when one replaces the assumption that the girth of $G$ is larger than $g$ with the weaker assumption that $G$ does not contain a cycle of length $2\lfloor \frac{g}{2} \rfloor$.

The study of circumference (the length of a longest cycle) of graphs with given girth and minimum degree was initiated by Ore~\cite{Or67} and has since attracted the attention of many researchers, see~\cite{ElMe00, SuVe08} and references therein. Several years ago, this study culminated in a result of Sudakov and Verstra{\"e}te~\cite{SuVe08}, who proved that every graph with minimum degree at least~$k$ whose girth exceeds~$g$ contains a cycle of length $\Omega(k^{\lfloor g / 2 \rfloor})$. Actually, it is proved in~\cite{SuVe08} that if $\HH = \{H\}$, where $H$ is a bipartite graph containing a cycle, then every $\HH$-free graph $G$ with average degree $k$ contains a family of cycles whose lengths are $n_\HH(\delta k)$ consecutive even integers, for some positive constant~$\delta$. In particular, every such $G$ contains a cycle of length at least $n_\HH(\delta k)$. As every graph with average degree $k$ contains a subgraph with minimum degree at least $k/2$, the last statement is also a straightforward corollary of our Theorem~\ref{thm:cycle}.

It is perhaps a good point to discuss yet another interpretation of our results, related to robustness of graph properties. The general approach of \emph{robustness}, explicitly promoted in~\cite{KrLeSu}, suggests to investigate whether graph theoretic properties and statements typically remain valid under taking random subgraphs -- which would then indicate that they are robust under (massive) random deletions. For example, it is elementary to prove that any graph $G$ of minimum degree at least $k$, $k\ge 2$, contains a cycle of length at least $k+1$. Corollary~\ref{cor:cycle} shows that this property typically stays with the random subgraph $G_p$ of $G$, even when the edge probability $p$ is only a notch above the critical probability $p^*=\frac{1}{k}$. Moreover, if in addition $G$ is assumed to be $\HH$-free, then not only $G$ contains deterministically a cycle of length at least $n_\HH(\delta k)$, but the $p$-random subgraph $G_p$ of $G$ retains this property with probability exponentially (in $k$) close to 1, even for $p=\frac{1+\eps}{k}$. Thus the property of containing long cycles is robust under taking random subgraphs. This complements in a substantial way the qualitative statement of~\cite[Theorem~1.3]{KrLeSu} (see also~\cite{Ri}), which says that if the minimum degree of $G$ is at least $k$ and $p = \frac{\omega(1)}{k}$, then with probability tending to $1$ as $k \to \infty$, $G_p$ contains a cycle of length $k - o(k)$.

Even though the existence of a cycle of length $\ell$ in a graph immediately implies the existence of a path of length $\ell-1$, we give a separate, much shorter, argument to prove that if $G$ is an $\HH$-free graph with minimum degree at least $k$ and $p \ge \frac{1+\eps}{k}$, then with probability close to $1$, the random graph $G_p$ contains a path of length $n_\HH(\delta k)$ for some positive contant $\delta$, see Theorem~\ref{thm:path} below. Our proof of Theorem~\ref{thm:path} is a fairly straightforward adaptation of the argument given in~\cite{KrSu}.

\begin{thm}
  \label{thm:path}
  Let $\eps \in (0,1)$, let $k$ be an integer, let $\HH$ be a good family of graphs, and let $\ell$ be an integer satisfying
  \begin{equation}
    \label{eq:ell-path}
    \frac{\ex(6\ell/\eps, \HH)}{\ell} \le \frac{k}{2}.
  \end{equation}
  If $G$ is an $\HH$-free graph with minimum degree at least $k$ and $p \ge \frac{1+\eps}{k}$, then
  \[
  \Pr\big(\text{$G_p$ contains a path of length $\ell$}\big) \ge 1 - 3\exp\left(-\frac{\eps^3 k}{300}\right).
  \]
\end{thm}

Note the more explicit, as compared to Theorem~\ref{thm:cycle}, dependence of $\ell$ on $\HH$, $k$, and $\eps$. In order to see that Theorem~\ref{thm:path} implies that $G_p$ typically contains a path of length $n_\HH(\delta k)$, note that~\eqref{eq:ell-path} is satisfied when $\ell = n_\HH(\delta k)$ and $\delta$ is sufficiently small (as a function of $\eps$), cf.\ the definition of $n_\HH(k)$ below~\eqref{eq:n-lower} and Lemma~\ref{lemma:Turan}. In particular, observe that when the family $\HH$ is empty, then \eqref{eq:ell-path} is satisfied when $\ell = \eps^2k/36$, and hence Theorem~\ref{thm:path} implies that with probability very close to $1$, $G_p$ contains a path of length $\eps^2k/36$. This is somewhat weaker than \cite[Theorem~4]{KrSu}, which asserts that under the same assumptions, i.e., $\delta(G) \ge k$ and $p \ge \frac{1+\eps}{k}$, with probability tending to $1$ as $k \to \infty$, the random graph $G_p$ contains a path of length $\frac{\eps^2 k}{5}$, but is still optimal up to a constant factor as when $p = \frac{1+\eps}{k}$, then a.a.s.\ the longest path in $G(k+1,p)$ has length at most $2 \eps^2 k$, see, e.g.,~\cite[Theorem~5.17]{JaLuRu00}.

At the heart of our proofs of Theorems~\ref{thm:cycle} and \ref{thm:path} lies the analysis of the execution of the depth-first search algorithm on the random graph $G_p$. This approach to investigating the properties of random graphs near the threshold for the appearance of the giant component was considered in~\cite{KrSu}, and our work draws heavily from there. Having said that, we would like to stress the fact that our proof of Theorem~\ref{thm:cycle} is not a mere adaptation of the arguments from~\cite{KrSu} and employs several novel ideas.

The remainder of the paper is organized as follows. In Section~\ref{sec:prelim}, we introduce some notational conventions and list several auxiliary lemmas that we will refer to in the proofs of our main results. In Section~\ref{sec:DFS}, we describe the depth-first search algorithm and list some of its properties for later reference. Sections~\ref{sec:proof-path} and~\ref{sec:proof-cycle} contain proofs of Theorems~\ref{thm:path} and~\ref{thm:cycle}, respectively. We close with some concluding remarks and open problems, in Section~\ref{sec:remarks}.

\section{Preliminaries}

\label{sec:prelim}

\subsection{Notation}

\label{sec:notation}

We use standard graph theoretic notation. In particular, given a graph $G$, we denote its vertex set by $V(G)$ and the number of edges by $e(G)$. Given a set $A \subseteq V(G)$, we denote the subgraph of $G$ induced by the set $A$ by $G[A]$. For $v \in V(G)$ and $A \subseteq V(G)$, we denote the number of neighbors of $v$ (the degree of $v$ in $G$) and the number of neighbors of $v$ in the set $A$ by $\deg_G(v)$ and $\deg_G(v,A)$, respectively. For two disjoint sets $A, B \subseteq V(G)$, we write $e_G(A,B)$ to denote the number of edges of $G$ with one endpoint in $A$ and one endpoint in $B$.

We shall now introduce the notion of \emph{excess edges}, which will play a crucial role in the proof of Theorem~\ref{thm:cycle}. Suppose that $G$ is an $n$-vertex graph with $r$ connected components of sizes $n_1, \ldots, n_r$, respectively. Since each connected component contains a spanning tree, then clearly $e(G) \ge n_1 - 1 + \ldots + n_r - 1 = n - r$. The number of excess edges of $G$, which we denote by $\exc(G)$, is the difference between $e(G)$ and this trivial lower bound. In other words, we let
\begin{equation}
  \label{eq:excess-comp}
  \exc(G) = e(G) - |V(G)| + \text{\#connected components of $G$}.
\end{equation}
Since adding an edge to a graph is easily seen no to decrease its excess, it follows that if $G$ is a graph and $U$ is a (not necessarily induced) subgraph of $G$, then
\begin{equation}
  \label{eq:excessUleG}
  \exc(U) \le \exc(G).
\end{equation}
We will use this simple observation in the proof of our main result.

Finally, let us remark that we will repeatedly omit rounding symbols whenever they are not crucial and treat large numbers as integers.

\subsection{Tools}
\label{sec:tools}

In our proofs, we will use the following standard estimate on tail probabilities of the binomial distribution, see, e.g., \cite[Appendix~A]{AlSp}.

\begin{lemma}
  \label{lemma:lrg-dev}
  Let $n$ be a positive integer, let $p \in [0,1]$, and let $X \sim \Bin(n,p)$.
  \begin{enumerate}[(i)]
  \item
    \label{item:lrg-dev-1}
    (Chernoff's inequality) For every positive $a$ with $a \le np/2$,
    \[
    P(|X - np| > a) < 2\exp\left(-\frac{a^2}{4np}\right).
    \]
  \item
    \label{item:lrg-dev-2}
    For every positive $\kappa$,
    \[
    P(X > \kappa np) \le \left(\frac{e}{\kappa}\right)^{\kappa np}.
    \]
  \end{enumerate}
\end{lemma}

In the proof of Theorem~\ref{thm:cycle}, we will use the following simple estimates on the rate of growth of the Tur{\'a}n function, Lemmas~\ref{lemma:Turan} and \ref{lemma:Turan-avg} below.

\begin{lemma}
  \label{lemma:Turan}
  Let $\HH$ be an arbitrary family of graphs and let $m$ and $n$ be integers with $n \ge m \ge 2$. Then
  \[
  \ex(n,\HH) \le \left(\frac{n-1}{m-1}\right)^2 \cdot \ex(m,\HH).
  \]
\end{lemma}
\begin{proof}
  Let $G$ be an $\HH$-free graph with $n$ vertices and $\ex(n,\HH)$ edges. The subgraph of $G$ induced by any set of $m$ vertices has at most $\ex(m,\HH)$ edges and therefore,
  \[
  \binom{n}{m} \cdot \ex(m,\HH) \ge \binom{n-2}{m-2} \cdot \ex(n,\HH),
  \]
  which easily implies the claimed inequality.
\end{proof}

\begin{lemma}
  \label{lemma:Turan-avg}
  Let $\HH$ be a good family of graphs and let $m$ and $n$ be integers with $n \ge m \ge 2$. Then
  \[
  \frac{\ex(m,\HH)}{m} \le 2 \cdot \frac{\ex(n,\HH)}{n}.
  \]
\end{lemma}
\begin{proof}
  The claimed inequality is an immediate consequence of the simple observation that, since each graph in the family $\HH$ is connected, an $\HH$-free graph on $n$ vertices can be obtained by taking $\lfloor \frac{n}{m} \rfloor$ vertex-disjoint copies of an $\HH$-free graph with $m$ vertices and $\ex(m, \HH)$ edges, thus implying that $\ex(n,\HH) \ge \lfloor \frac{n}{m} \rfloor \cdot \ex(m,\HH) \ge \frac{n}{2m} \cdot \ex(m,\HH)$.
\end{proof}

\section{Depth-first search algorithm}

\label{sec:DFS}

At the heart of our approach lies the \emph{depth-first search algorithm} (DFS algorithm for short), which is a well-known graph exploration method. We briefly describe it below.

The DFS algorithm takes as input a finite graph $G$ on a vertex set $V$ and, after visiting all vertices of $G$, outputs a rooted spanning forest $F$ of $G$ such that the connected components of~$F$ are the connected components of $G$. Since we are interested not only in the connectivity structure of $G$ but also in its cycles, we make our algorithm eventually examine all the edges of $G$, not only those that belong to $F$. Consequently, our version of the DFS algorithm runs in two phases. In the first phase, the algorithm discovers the connected components of $G$ and constructs the spanning forest $F$. In the second phase, it examines the remaining edges of $G$, whose both endpoints lie in the same tree of $F$ (connected component of $G$).

At all times, the algorithm maintains a partition of $V$ into three
sets $S$, $T$, and $U$. At any given moment, the set $S$ contains
vertices whose exploration is complete (i.e., whose neighborhood in
$F$ has been fully determined), $T$ is the set of vertices that have
not yet been visited, and $U$ consists of vertices that are being
explored. The vertices in $U$ are kept in a stack, that is, a
last-in-first-out data structure. The algorithm starts with $T = V$
and $S = U = \emptyset$ and, as it examines the edges of~$G$, the
vertices of $G$ are moved from $T$ to $U$ and from $U$ to $S$. The
algorithm switches from the first to the second phase when $S = V$
and $U = T = \emptyset$. Finally, the procedure terminates when all
the edges of $G$ have been examined.

We assume that the set $V$ is equipped with some canonical linear
order $\ple$. The first phase of the execution of the DFS algorithm
can be divided into $2|V|$ rounds. At the beginning of each round,
the algorithm checks whether or not $U$ is empty. If $U =
\emptyset$, then the $\ple$-smallest vertex of~$T$ is moved to $U$.
Otherwise, the algorithm considers the top (most recently added)
element $v$ of~$U$ and queries whether $\{v, w\}$ is an edge of $G$
for some $w \in T$, examining the vertices of $T$ from the
$\ple$-smallest to the $\ple$-largest. If $v$ does have a neighbor
$w$ in $T$, then the $\ple$-smallest such $w$ is moved from $T$ to
$U$, becoming its top element; no more queries about whether or not
$\{v, w'\}$ is an edge of $G$ for some $w' \in T$ with  $w \pl w'$
are asked in this round. Otherwise, if the top element $v$ of~$U$
has no neighbors in $T$, then $v$ is moved to $S$. Finally, unless
$S = V$, the algorithm proceeds to the next round. Observe that in
each round exactly one vertex moves, either from $T$ to $U$ or from~$U$ to $S$. Since at the end of the first phase all vertices have
made their way from $T$ to $S$, passing through $U$, the number of
rounds is indeed $2|V|$.

At the end of the first phase, the DFS algorithm has constructed a rooted spanning forest $F$ of the input graph $G$. The root of each tree in $F$ is the first vertex of it that was moved from $T$ to $U$. The edges of $F$ are precisely those pairs $\{v,w\} \in G$ such that at some point during the execution of the algorithm, $v$ was the top element of $U$ and $w$ was the $\ple$-smallest neighbor of $v$ in $T$. It is possible that at the end of the first phase, some pairs of vertices have not yet been queried. Note that if $\{v,w\}$ is such a pair, then necessarily $v$ and $w$ belong to the same tree component of $F$ found by the algorithm and, moreover, $v$ is a predecessor of $w$ (or vice-versa) in this rooted tree, as otherwise the algorithm would have queried $\{v,w\}$. In particular, each such $v$ and $w$ are connected by a unique path in this tree. We denote the length of this path by $\len(v,w)$. In order to complete the exploration of all edges of $G$, in the second phase of its execution, the algorithm queries all the previously not queried pairs $\{v,w\}$, ordered according to the value of $\len(v,w)$, from the smallest to the largest. We break ties arbitrarily, that is, the pairs with the same value of $\len(\cdot,\cdot)$ are queried in an arbitrary order.

Finally, we list several properties of the DFS algorithm for future reference:
\begin{enumerate}
\item
  \label{item:DFS-C}
  The algorithm starts exploring a connected component $C$ of $G$ at the moment the $\ple$-smallest vertex of $C$ is moved into the (empty beforehand) set $U$ and completes discovering $C$ when $U$ becomes empty again. At the moment when $C$ is fully discovered, all of its vertices are in $S$.
\item
  \label{item:DFS-U-grows}
  If $T \neq \emptyset$, then every positively answered query increases the size of $U$ by one.
\item
  \label{item:DFS-ST}
  At any stage, the algorithm has queried all pairs $\{v,w\}$ with $v \in S$ and $w \in T$ and found out that $\{v, w\} \not\in G$.
\item
  \label{item:DFS-path}
  The set $U$ always spans a path in $G$.
\end{enumerate}

In several of our proofs, we will analyze the execution of the DFS algorithm on some (random) subgraph of a given graph $G$. Our assumption will be that the base graph $G$ is known to the algorithm and it thus asks queries only about the edges of $G$. We will often use the fact that our graph exploration algorithm implicitly defines a bijection $\varphi$ between the set of all $\{0,1\}$-sequences of length $e(G)$ and the family of all subgraphs of $G$, which pairs subgraphs with $m$ edges with sequences containing exactly $m$ ones. This bijection is defined as follows: Given a graph $G' \subseteq G$, we run the DFS algorithm with input $G'$. We start with $\varphi(G')$ being the empty sequence and each time the algorithm queries whether some $\{v, w\} \in G$ is an edge of $G'$, we append to $\varphi(G')$ the answer to this query, i.e., $1$ if $\{v, w\} \in G'$ and $0$ otherwise. We will sometimes say that $G'$ is represented by the sequence $\varphi(G')$.

Conversely, any $\{0,1\}$-sequence $(X_i)$ of length $e(G)$ represents a subgraph $G'$ of $G$ that is described as follows. We run the DFS algorithm and each time it queries whether some pair $\{v, w\} \in G$ is an edge of $G'$, we let $\{v, w\} \in G'$ if $X_i = 1$ and $\{v, w\} \not\in G'$ otherwise, where $i$ is the number of the query (the algorithm examines every edge of $G$ exactly once). In particular, if $(X_i)$ is a sequence of i.i.d.\ Bernoulli random variables with success probability $p$, then $G'$ is the random subgraph $G_p$ of $G$.

\section{Proof of Theorem~\ref{thm:path}}

\label{sec:proof-path}

Let $G$ be a finite $\HH$-free graph with minimum degree at least $k$ and assume that $p \ge \frac{1+\eps}{k}$. Without loss of generality, we may assume that $\ell$ is the largest integer satisfying~\eqref{eq:ell-path}. This quantity is well defined as the assumptions that $\HH$ is finite and that each graph in $\HH$ contains a cycle imply that $\ex(n,\HH) = \Omega(n^{1+\eta})$ for some positive constant $\eta$, since for every integer $g \ge 3$, there are well-known constructions of $n$-vertex graphs with girth exceeding $g$ and $\Omega(n^{1 + 1/(g-1)})$ edges, see, e.g., \cite{Er59}. In particular,
\[
\frac{36(\ell+1)}{2\eps^2} \ge \frac{\ex\big(6(\ell+1)/\eps,\HH\big)}{\ell+1} \ge \frac{k}{2},
\]
and hence
\begin{equation}
  \label{eq:ell-path-lower}
  \ell > \frac{\eps^2 k}{36}-1.
\end{equation}
Note also that the number $n$ of vertices of $G$ satisfies $nk \le 2 \cdot \ex(n,\HH)$, see~\eqref{eq:n-lower}, which implies that $n > 6\ell / \eps$, as otherwise Lemma~\ref{lemma:Turan-avg} and~\eqref{eq:ell-path} would yield
\[
\frac{\ex(n,\HH)}{n} \le 2 \cdot \frac{\ex(6\ell/\eps, \HH)}{6\ell/\eps} \le \frac{\eps k}{6} < \frac{k}{2},
\]
a~contradiction.

Consider a random $\{0,1\}$-sequence $X$ of length $e(G)$ whose all positions are i.i.d.\ Bernoulli random variables with success probability $p$.  We will show that if we run the DFS algorithm on the random subgraph of $G$ represented by $X$, then with probability at least $1 - 3\exp(-\eps^3 k/300)$, at the moment when $|S \cup U|$ reaches $6\ell/\eps$, the set $U$ contains at least $\ell$ elements. (To see that such moment must occur, recall that in each round of the DFS algorithm, the size of $S \cup U$ does not change or increases by one and that eventually $|S| = n > 6\ell/\eps$.) Since $U$ spans a path in this random subgraph, see property~(\ref{item:DFS-path}) of the DFS algorithm, the assertion of the theorem will follow.

Consider the moment when $|S \cup U|$ reaches $6\ell/\eps$ and suppose to the contrary that $|U| < \ell$. Let $Q$ be the number of queries about edges of $G$ that have been asked so far. By property~(\ref{item:DFS-ST}) of the DFS algorithm,
\[
\begin{split}
  Q & \ge e(S,T) = \sum_{v \in S} \deg(v) - 2e(S) - e(S,U) \ge |S| \cdot k - 2e(S \cup U) \\
  & \ge \left(\frac{6\ell}{\eps} - \ell\right) \cdot k - 2 \cdot \ex(6\ell / \eps,\HH) \ge \frac{6k\ell}{\eps} -2 k \ell.
\end{split}
\]
Let $P$ be the number of positively answered queries among the first $(6/\eps - 2)k\ell$ queries about edges of $G$. Since $P$ is a binomial random variable, it follows from Chernoff's inequality (Lemma~\ref{lemma:lrg-dev}) that with probability at least $1 - 2\exp(-\eps\ell/8)$, which, by~\eqref{eq:ell-path-lower}, is at least $1 - 3\exp(-\eps^3 k/300)$,
\[
P > \left(1+\frac{\eps}{2}\right)\left(\frac{6}{\eps} - 2\right)\ell > \frac{6\ell}{\eps},
\]
where in the last inequality we used the assumption that $\eps < 1$. Finally, since $T$ is still non-empty (as $n > 6\ell/\eps = |S \cup U|$), it follows from property~(\ref{item:DFS-U-grows}) of the DFS algorithm that $|S \cup U| \ge P$, a~contradiction. \qed

\section{Proof of Theorem~\ref{thm:cycle}}

\label{sec:proof-cycle}

\subsection{Proof outline}

Let us start by giving a brief outline of the proof. Suppose that $G$ is an $\HH$-free graph with $n$ vertices and minimum degree at least $k$, for some sufficiently large integer $k$, and that $p \ge \frac{1+\eps}{k}$. The key step in the proof is to show that with high probability, the random graph $G_p$ contains $\Omega(n)$ excess edges. To this end, we first prove that with probability $\Omega(1)$, a positive proportion of the vertices of $G_p$ lies in large connected components (Theorem~\ref{thm:pr-lrg-comp}). Second, we observe that it is extremely unlikely that there is a set $B$ of $\Omega(n)$ vertices that belong only to large connected components in $G_p$ but $G[B]$ has merely $o(|B| \cdot k)$ edges. It then follows from a standard double exposure argument that with positive probability, $G_p$ has $\Omega(n)$ excess edges (Theorem~\ref{thm:excess-edges}). Finally, we note that the number of excess edges in a random graph is tightly concentrated around its expectation (Proposition~\ref{prop:excess-conc}) and therefore, $\exc(G_p) = \Omega(n)$ with very high probability (Corollary~\ref{cor:excess}). This means that when we run the DFS algorithm on a typical $G_p$, then the number of queries about edges of $G$ asked in the second phase of its execution is $\Omega(nk)$. Since the graph $G$ is $\HH$-free, which implies upper bounds on the densities of induced subgraphs of $G$, at least $\Omega(k^2)$ of these queries are about edges $\{u,v\} \in G$ such that $\lambda(u,v) = \Omega\big(n_\HH(k)\big)$. With high probability, one of these pairs is an edge of $G_p$; this edge closes a cycle of length $\Omega\big(n_\HH(k)\big)$ in $G_p$.

\subsection{Bounding the number of excess edges}

As mentioned above, the key ingredient in our proof of Theorem~\ref{thm:cycle} is the fact that if $G$ is a graph with minimum degree at least $k$ and $p \ge \frac{1+\eps}{k}$, then with high probability, the random graph $G_p$ contains $\Omega(|V(G)|)$ excess edges, see Corollary~\ref{cor:excess} below. This statement will be an immediate consequence of the following two facts. First, in Proposition~\ref{prop:excess-conc}, using a martingale concentration result, we prove that for arbitrary graph $G$ and probability $p$, the number of excess edges in the $p$-random subgraph of $G$ is concentrated around its expectation. Second, in Theorem~\ref{thm:excess-edges}, we show that if $\delta(G) \ge k$ and $p \ge \frac{1+\eps}{k}$, then this expectation is at least $c \cdot |V(G)|$ for some positive constant $c$.

\begin{prop}
  \label{prop:excess-conc}
  Let $G$ be an arbitrary graph and let $p \in [0,1/2]$. Let $\mu$ be the expected number of excess edges in the random graph $G_p$. Then for every $\beta \in (0,1]$,
  \[
  \Pr\big(|\exc(G_p) - \mu| \ge \beta p e(G)\big) \le 2\exp\left( - \frac{\beta^2pe(G)}{3}\right).
  \]
\end{prop}
\begin{proof}
  Let $m = e(G)$ and fix an arbitrary ordering $e_1, \ldots, e_m$ of the edges of $G$. For each $i$, let $X_i$ be the indicator random variable of the event $e_i \in G_p$. Fix an $i$, let $A$ be an arbitrary subset of $\{e_1, \ldots, e_{i-1}\}$, and let $\cA_{i,A}$ denote the event that $G_p \cap \{e_1, \ldots, e_{i-1}\} = A$. Following McDiarmid~\cite{Mc}, we define $g_{i,A} \colon \{0,1\} \to \NN$ by
  \[
  g_{i,A}(x) = \Ex\big[\exc(G_p) \mid \cA_{i,A} \text{ and } X_i = x\big] - \Ex\big[\exc(G_p) \mid \cA_{i,A}\big].
  \]
  The function $g_{i,A}$ measures how much the expected number of excess edges in $G_p$ changes when it is revealed whether $e_i$ is or is not an edge of $G_p$. Observe crucially that the function $\exc(\cdot)$ is edge Lipschitz, i.e., adding or deleting a single edge to/from a graph changes the number of excess edges by at most one. It follows that $|g_{i,A}(1) - g_{i,A}(0)| \le 1$ for all $i$ and $A$. Applying~\cite[Theorem~3.9]{Mc} to the sequence $(X_i)$ with $f = \exc$ and $t = \beta p e(G)$, noting that $b = \maxdev \le 1$ and $\hat{r}^2 \le e(G)$, we get
  \[
  \Pr\big(|\exc(G_p) - \mu| \ge \beta p e(G)\big) \le 2 \exp\left( - \frac{(\beta p e(G))^2}{2pe(G) + \beta p e(G)} \right) \le 2 \exp\left( - \frac{\beta^2pe(G)}{3}\right).\qedhere
  \]
\end{proof}

Our proof of Theorem~\ref{thm:excess-edges} will use the following fact, which is implicit in many earlier works on the phase transition in $G(n,p)$, see, e.g., \cite{Bo01, JaLuRu00}. Our proof here, which we include for the sake of completeness, follows the approach of~\cite{KrSu}.

\begin{thm}
  \label{thm:pr-lrg-comp}
  Let $\eps \in (0, 1/3)$, let $k$ be an arbitrary integer, let $G$ be a graph with minimum degree at least $k$, and let $v$ be an arbitrary vertex of $G$. If $p \ge \frac{1+\eps}{k}$, then
  \[
  \Pr\big(\text{the connected component of $v$ in $G_p$ has at least $\eps k / 2$ vertices} \big) > \eps/6.
  \]
\end{thm}
\begin{proof}
  Let $G$ be a graph with minimum degree at least $k$, let $v$ be an arbitrary vertex of $G$, and fix some arbitrary linear order $\ple$ on $V$ whose smallest element is $v$. Assume that $p \ge \frac{1+\eps}{k}$ and consider the $p$-random $\{0,1\}$-sequence $X$ of length $e(G)$. We will show that if we run the DFS algorithm on the random subgraph of $G$ represented by $X$, then with probability at least $\eps/6$, the following event holds:

  \noindent
  \begin{center}
    \begin{tabular}{rp{0.85\textwidth}}
      $\cA$: & From the moment $v$ is moved to $U$, the set $U$ does not become empty until $|S| \ge \eps k/2$.
    \end{tabular}
  \end{center}
  This clearly implies the assertion of the theorem, see property~(\ref{item:DFS-C}) of the DFS algorithm.

  With the aim of estimating the probability of $\cA$, let us define for each $i \in \{1, \ldots, e(G)\}$,
  \[
  Y_i =
  \begin{cases}
    (1-\eps/2) \cdot k & \text{if $X_i = 1$},\\
    -1 & \text{if $X_i = 0$},
  \end{cases}
  \qquad
  \text{and}
  \qquad
  Z_i = (1-\eps/2) \cdot k + \sum_{j = 1}^i Y_j.
  \]
  We claim that if $Z_i > 0$ for all $i \le e(G)$, then $\cA$ holds. Indeed, suppose that the sequence $X$ is such that $Z_i > 0$ for all $i$ and consider the execution of the DFS algorithm on the random subgraph of $G$ defined by $X$. Suppose that $\cA$ does not hold, that is, the set $U$ becomes empty when $|S| < \eps k / 2$. Recall that initially $U$ contains one element (the vertex $v$) and the size of $U$ increases by one every time $X_i = 1$ for some $i$. Moreover, at the time a vertex $w$ is moved from $U$ to $S$, the DFS algorithm has already queried all the $\deg(w,T)$ edges of $G$ that connect $w$ and $T$ and got a negative answer for each of these queries. Finally, note that clearly $\deg_G(w,T) \ge k - |S| - |U| \ge (1-\eps/2) \cdot k$. In particular, the algorithm gets at least $(1-\eps/2)k$ negative answers to queries about edges of $G$ incident to $w$ before moving $w$ from $U$ to $S$. These three facts readily imply that if $\cA$ does not hold, then $Z_i \le 0$ for some $i$.

  Finally, we estimate the probability that $Z_i > 0$ for all $i$. To this end, let $\alpha$ be the smallest positive real that satisfies
  \begin{equation}
    \label{eq:alpha}
    f(\alpha) := p \cdot \alpha^{(1-\eps/2) \cdot k} + (1-p) \cdot \alpha^{-1} = 1.
  \end{equation}
  We now argue that $\alpha \le e^{-\frac{\eps}{4k}}$. To this end, observe that $\lim_{\alpha \to 0^+} f(\alpha) = \infty$ and, recalling that $\eps < 1/3$ and that $e^x \le 1+x+x^2$ if $|x| \le 1$,
  \[
  \begin{split}
    f\left(e^{-\frac{\eps}{4k}}\right) & = p \cdot e^{-\frac{\eps}{4} \cdot \left(1-\frac{\eps}{2}\right)} + (1-p) \cdot e^{\frac{\eps}{4k}} \\
    & \le p \cdot \left( 1 - \frac{\eps}{4} + \frac{3\eps^2}{16} \right) + (1-p) \cdot \left( 1 + \frac{\eps}{4k} + \frac{\eps^2}{16k^2} \right) \\
    & = 1 - \frac{\eps}{4} \cdot \left(p - \frac{1}{k} - \frac{3\eps p}{4}\right) - \frac{\eps}{4 k} \left(p - \frac{\eps (1-p)}{4k}\right) < 1.
  \end{split}
  \]
  Now, let $Z_0 = (1-\eps/2)k$ and for each $i$ with $0 \le i \le e(G)$, let $M_i = \alpha^{Z_i}$. Since $\alpha$ satisfies~\eqref{eq:alpha}, the sequence $(M_i)$ is a martingale. Furthermore, let
  \[
  T = \min\{i \colon Z_i \le 0\} \qquad \text{and} \qquad \Mb_i = M_{\min\{i,T\}}
  \]
  and observe that the sequence $(\Mb_i)$ is also a martingale, as $T$ is a stopping time with respect to~$(M_i)$. Therefore,
  \[
  \Ex[\Mb_{e(G)}] = \Ex[\Mb_0] = \Ex[M_0] = \Ex[\alpha^{Z_0}] = \alpha^{\left(1-\frac{\eps}{2}\right)k} \le e^{-\frac{\eps}{4}\left(1-\frac{\eps}{2}\right)}.
  \]
  On the other hand, since $\alpha < 1$ and therefore $\Mb_{e(G)} \ge 1$ whenever $Z_i \le 0$ for some $i$, then
  \[
  \Ex[\Mb_{e(G)}] \ge \Pr(\text{$Z_i \le 0$ for some $i$}).
  \]
  It follows that, recalling again that $e^x \le 1 + x + x^2$ if $|x| \le 1$,
  \[
  \Pr(\text{$Z_i > 0$ for all $i$}) \ge 1 - e^{-\frac{\eps}{4}\left(1-\frac{\eps}{2}\right)} \ge \frac{\eps}{4}\left(1-\frac{\eps}{2}\right) - \frac{\eps^2}{16} > \frac{\eps}{6}.\qedhere
  \]
\end{proof}

\begin{thm}
  \label{thm:excess-edges}
  For every positive $\eps$, there exists a positive constant $c$ such that the following holds. Let $k$ be a sufficiently large integer and let $G$ be an $n$-vertex graph with minimum degree at least $k$. If $p \ge \frac{1+\eps}{k}$, then the expected number of excess edges in $G_p$ is at least $c n$.
\end{thm}
\begin{proof}
  Let us first choose some parameters. Let
  \[
  \gamma = \frac{\eps}{24} \qquad \text{and} \qquad \delta = \frac{\gamma^2}{4e^4}
  \]
  and assume that $k \ge \frac{32}{\delta \eps^2}$. Without loss of generality, we may also assume that $\eps \le \frac{2}{3}$. Let $G$ be a~graph with minimum degree at least $k$ and suppose that $p \ge \frac{1+\eps}{k}$. We will expose the edges of $G_p$ in two rounds. To this end, let $p_1 = \frac{1 + \eps/2}{k}$, let $p_2$ satisfy $(1-p_1)(1-p_2) = 1-p$, and observe that $p_2 \ge p - p_1 \ge \frac{\eps}{2k}$. Consider the following two events in $G_{p_1}$:

  \noindent
  \begin{center}
    \begin{tabular}{rp{0.85\textwidth}}
      \smallskip
      $\cA$: & at least $\gamma n$ vertices of $G_{p_1}$ lie in components of size at least $\eps k /4$. \\
      $\cB$: & some $B \subseteq V(G)$ with $|B| \ge \gamma n$ and $e(G[B]) \le \delta |B| k$ satisfies $e(G_{p_1}[B]) \ge |B|/2$.
    \end{tabular}
  \end{center}

  \noindent
  Suppose now that $\cA$ holds. In this case, there is a set $A \subseteq V(G)$ of at least $\gamma n$ vertices such that each connected component of $G_{p_1}[A]$ has at least $\eps k /4$ vertices and therefore,
  \[
  e(G_{p_1}[A]) \ge \left(1 - \frac{4}{\eps k}\right) |A| > \frac{|A|}{2}.
  \]
  Assume furthermore that $\cB$ does not hold. Then necessarily $e(G[A]) > \delta |A| k$. Observe that each edge of $G \setminus G_{p_1}$ belongs to $G_{p_2}$ with probability $p_2$. Hence, considering separately the cases $e(G_{p_1}[A]) \ge \frac{\delta |A| k}{2}$ and $e(G_{p_1}[A]) < \frac{\delta |A| k}{2}$, we obtain the following bound:
  \begin{equation}
    \label{eq:excess-A-notB}
    \begin{split}
      \Ex\big[e(G_p[A]) \mid \cA \wedge \neg \cB\big] & \ge \Ex\left[e(G_{p_1}[A]) + p_2 \cdot \big(e(G[A]) - e(G_{p_1}[A])\big) \mid \cA \wedge \neg \cB\right] \\
      & \ge \min \left\{ \frac{\delta|A|k}{2}, \left(1- \frac{4}{\eps k} \right)|A| + p_2 \cdot \frac{\delta|A|k}{2} \right\} \\
      & \ge \left(1 - \frac{4}{\eps k} + \frac{\eps\delta}{4}\right)|A| \ge \left(1+\frac{\eps\delta}{8}\right)|A| \ge |A| + \frac{\eps \delta \gamma}{8}n.
    \end{split}
  \end{equation}
  Consequently, since $\exc(G_p) \ge \exc(G_p[A])$ by~(\ref{eq:excessUleG}), the expected number of excess edges in $G_p$, conditioned on $\cA$ and $\neg \cB$ holding simultaneously, satisfies
  \[
  \Ex\big[\exc(G_p) \mid \cA \wedge \neg \cB\big] \ge \Ex \big[\exc(G_p[A]) \mid \cA \wedge \neg \cB \big] \ge \frac{\eps\delta\gamma}{8} n.
  \]
  It is therefore enough to prove the following claim.

  \begin{claims}
    $\Pr(\cA \wedge \neg \cB) \ge \gamma/2$.
  \end{claims}
  Let $L$ be the number of vertices of $G_{p_1}$ that lie in components of size at least $\eps k / 4$. It follows from Theorem~\ref{thm:pr-lrg-comp} that $\Ex[L] \ge 2\gamma n$ and hence, $\Pr(\cA) \ge \gamma$, as clearly $L \le n$. With the aim of estimating the probability of $\cB$, fix an arbitrary set $B$ of $b$ vertices, where $b \ge \gamma n$, such that $e(G[B]) \le \delta b k$. It follows from Lemma~\ref{lemma:lrg-dev} that
  \[
  \Pr\big(e(G_{p_1}[B]) \ge b/2\big) \le \Pr\big(\Bin(\delta b k, 2/k) \ge b/2\big) \le (4 e \delta)^{b/2}.
  \]
  Consequently,
  \[
  \Pr(\cB) \le \sum_{b \ge \gamma n} \binom{n}{b} (4 e \delta)^{b/2} \le \sum_{b \ge \gamma n} \left(\frac{en}{b}\right)^b (4 e \delta)^{b/2} \le \sum_{b \ge \gamma n} \left(\frac{4 e^3 \delta}{\gamma^2}\right)^{b/2} = \sum_{b \ge \gamma n} e^{-b/2} < \gamma / 2,
  \]
  since $\gamma n \ge \gamma k \ge 4/\gamma$. This completes the proof of the claim and consequently, the proof of the theorem.
\end{proof}

\begin{cor}
  \label{cor:excess}
  For every positive $\eps$, there exists a positive constant $c$ such that the following holds. Let $k$ be a sufficiently large integer and let $G$ be an $n$-vertex graph with minimum degree at least $k$. If $p \ge \frac{1+\eps}{k}$, then with probability at least $1 - 2 \exp( - c^2n / 12 )$, the number of excess edges in $G_p$ is at least $c n$.
\end{cor}
\begin{proof}
  Let $c = \min\{1, c_{\ref{thm:excess-edges}}(\eps)/2\}$ and suppose that $k$ is sufficiently large, so that the assertion of Theorem~\ref{thm:excess-edges} is true. Without loss of generality, we may assume that $pe(G) \le 4n$ as otherwise by Chernoff's inequality (Lemma~\ref{lemma:lrg-dev}), $e(G_p) \ge 2n$ with probability at least $1 - \exp(-n/4)$, and consequently
  \[
  \exc(G_p) \ge e(G_p) - n \ge n \ge c n.
  \]
  It follows from Theorem~\ref{thm:excess-edges} that $\Ex[\exc(G_p)] \ge 2cn$ and therefore by Proposition~\ref{prop:excess-conc} with $\beta = \frac{cn}{pe(G)}$,
  \[
  \Pr(\exc(G_p) < cn) \le 2 \exp\left(-\frac{cn}{3pe(G)} \cdot cn\right) \le 2 \exp\left( - \frac{c^2n}{12}\right),
  \]
  as $pe(G) \le 4n$ by our assumption.
\end{proof}

\subsection{Finding a long cycle}

Let $G$ be an $\HH$-free graph with minimum degree at least $k$ and assume that $p \ge \frac{1+\eps}{k}$. Let $c = c_{\ref{cor:excess}}/3$, let $\delta = c^2/21$, and suppose that $k$ is sufficiently large, so that, in particular, the assertion of Corollary~\ref{cor:excess} holds. Without loss of generality, we may assume that $\eps \le 1$. Let $\ell$ be the largest integer satisfying
\begin{equation}
  \label{eq:ell-cycle}
  \ex(\ell, \HH) \le \frac{ck\ell}{20}.
\end{equation}
By the maximality of $\ell$, we have $\ex(\ell+1, \HH) > ck(\ell+1)/20$ and hence $\ell \ge n_\HH(ck/10) \ge n_\HH(\delta k)$. Consider the $p$-random $\{0,1\}$-sequence $X$ of length $e(G)$. Recall the definition of $\len(\cdot, \cdot)$ from Section~\ref{sec:DFS}. We will show that if we run the DFS algorithm on the random subgraph of $G$ represented by $X$, then with probability at least $1 - 4\exp(-c^2k/16)$, the following event holds:

\noindent
\begin{center}
  \begin{tabular}{rp{0.85\textwidth}}
    $\cA$: & At the moment when the DFS algorithm finishes discovering all the connected components, i.e., when $S = V$, the number of edges $\{v,w\} \in G$ such that $\{v, w\}$ has not been yet queried and $\len(v,w) \ge \ell$ is at least $ck\ell/2$.
  \end{tabular}
\end{center}

\noindent
This will clearly be sufficient, since it implies that each of the last $ck\ell/2$ edges of $G$ that are queried by our graph exploration algorithm closes a cycle of length at least $\ell$ and  with probability at least $1 - (1-p)^{ck\ell/2}$, one of these queries will be answered positively. Since, as noted above, $\ell \ge n_\HH(ck/10) \ge ck/10$, this probability is at least $1 - \exp(-c^2k/20)$.

Let $n$ be the number of vertices of $G$. We now argue that the following statement implies the claimed bound on the probability that $G_p$ contains a cycle of length at least $\ell$.

\begin{claims}
  The event $\cA$ contains the intersection of the following two events:
\end{claims}
\begin{enumerate}[(i)]
\item
  \label{item:cA-1}
  The number of excess edges in $G_p$ is at least $3cn$.
\item
  \label{item:cA-2}
  There are fewer than $3cn$ indices $i$ with $i > e(G) - ckn$ such that $X_i = 1$.
\end{enumerate}

Indeed, note that $n \ge k+1$, that, by Corollary~\ref{cor:excess}, (\ref{item:cA-1}) holds with probability at least $1 - 2\exp(-c^2n/2)$, and that, by Chernoff's inequality (Lemma~\ref{lemma:lrg-dev}), (\ref{item:cA-2}) holds with probability at least $1 - 2\exp(-cn/16)$.

It now suffices to prove the claim. Let $G' \subseteq G$ consist of pairs (edges of $G$) that have not been queried at the time when our graph exploration algorithm finishes discovering the connected components of $G_p$. Note that (\ref{item:cA-1}) and (\ref{item:cA-2}) imply that $e(G') \ge ckn$ as the rooted spanning forest $F$ of $G_p$ constructed by the DFS algorithm contains no excess edges. Recall from Section~\ref{sec:DFS} that each edge of $G'$ connects a vertex $w$ with its predecessor $v$ in one of the rooted trees forming $F$ and that $\len(v,w)$ denotes the length of the unique path joining $v$ and $w$ in this tree. Since the average degree of $G'$ is at least $2ck$ and the endpoints of each edge of $G'$ lie in the same connected component of $F$, there must be a rooted tree $R$ in $F$ such that the average degree of $G'[V(R)]$ is also at least $2ck$. We will now show that there are at least $ck\ell/2$ edges $\{v,w\} \in G'[V(R)]$ such that $\len(v,w) \ge \ell$.

To this end, for every subtree $R^*$ of $R$, let
\[
d(R^*) = \sum_{w \in V(R^*)} |\{v \in V(R) \colon \text{$v$ is a predecessor of $w$ in $R$ and } \{v,w\} \in G'\}|
\]
and note that $d(R) = e(G'[V(R)]) \ge ck|V(R)|$. We claim that $R$ contains a subtree $R^*$ with $\ell \le |V(R^*)| < 2\ell$ with $d(R^*) \ge ck\ell$. Indeed, we can construct such $R^*$ as follows. First, repeatedly delete from $R$ every full subtree\footnote{A subtree of a rooted tree that is induced by some vertex and all of its descendants.} $R'$ such that $d(R') < ck|V(R')|$ until there are no such $R'$ left. Since as a result of every such deletion, the ratio $d(R)/|V(R)|$ increases, we may assume that $R$ contains no such $R'$. We claim that $R$ still has at least $\ell$ vertices. Indeed, $G'[V(R)]$ is an $\HH$-free graph with average degree at least $2ck$ and for every $m \le \ell$, by~\eqref{eq:ell-cycle} and Lemma~\ref{lemma:Turan-avg},
\[
\frac{ex(m,\HH)}{m} \le 2 \cdot \frac{\ex(\ell,\HH)}{\ell} \le \frac{ck}{10}.
\]
We can obtain an $R^*$ with the claimed property by running the following simple recursive procedure on $R$: If $R$ has fewer than $2\ell$ vertices, then we let $R^* = R$. Otherwise, if the full subtree $R'$ rooted at one of the children of the root of $R$ has at least $\ell$ vertices, then we let $R = R'$ and work with $R'$, noting that $d(R') \ge ck|V(R')|$ by our assumption on the original tree. Else, $R$ has at least $2\ell$ vertices but each full subtree of $R$ attached to its root has fewer than $\ell$ vertices. In this case, we can easily obtain a subtree $R^*$ of $R$ that has the required properties by deleting some of these full subtrees.

Finally, let $R^+$ be the subtree of $R$ obtained from $R^*$ by adding to it the first (at most) $\ell$ vertices on the unique path from the root of $R^*$ to the root of $R$. By definition, if $\{v,w\}$ is an edge of $G'$ such that $\len(v,w) < \ell$ and $v$ is a predecessor of $w$, then $w \in V(R^*)$ implies that $v \in V(R^+)$. It follows that
\[
|\{\{v, w\} \in G' \colon \len(v,w) \ge \ell\}| \ge d(R^*) - e(G'[V(R^+)]) \ge ck\ell - \ex(3\ell, \HH) \ge ck\ell/2.
\]
To see the last inequality, note that by~\eqref{eq:ell-cycle} and Lemma~\ref{lemma:Turan},
\[
\ex(3\ell, \HH) \le 10 \cdot \ex(\ell, \HH) \le ck\ell/2.
\]
This completes the proof of the claim and consequently, the proof of the theorem. \qed

\section{Concluding remarks}

\label{sec:remarks}

\subsection{Dependence of $\ell$ on $\HH$ and $\eps$ and the function $n_\HH(\cdot)$}
\label{sec:dep-ell-eps}

In our proof of Theorem~\ref{thm:cycle}, we made no attempts to optimize the dependence of $\delta$ on $\eps$ and $\HH$. A careful analysis of the proof shows that $\delta$ has only polynomial dependence on $\eps$, that is, Theorem~\ref{thm:cycle} asserts the existence of a cycle of length $n_\HH(c_\HH \eps^{\alpha_\HH})$, where $c_\HH$ and $\alpha_\HH$ are positive constants depending only on $\HH$. It would be interesting to know whether in the setting of Corollary~\ref{cor:cycle}, i.e., when $\HH$ is the empty family, one can replace $\alpha_\HH$ with $2$, as in Theorem~\ref{thm:path}, where condition~\eqref{eq:ell-path} is satisfied for $\ell = \eps^2k/36$. 

A much more intriguing question is that about the relation between $n_\HH(k)$ and $n_\HH(\delta k)$. It is not difficult to prove that for every good family $\HH$ and positive constant $\delta$, we have $n_\HH(\delta k) \le (\delta + o(1)) n_\HH(k)$. Indeed, one may derive this estimate from Lemma~\ref{lemma:Turan} with $m = n_\HH(\delta k) - 1$ and $n = n_\HH(k)$. On the other hand, for most bipartite graphs $H$, since the known lower and upper bounds on $\ex(n,H)$ do not match, it is not even clear whether the ratio $n_{\{H\}}(k) / n_{\{H\}}(\delta k)$ is bounded for any constant $\delta < 1$.

\subsection{Regular graphs}
\label{sec:regular-G}

The proof of Theorem~\ref{thm:cycle} can be somewhat simplified (at least conceptually) when one makes the assumption that the graph $G$ is regular. Recall that the key step in our proof is establishing that with high probability, the graph $G_p$ contains $\Omega(n)$ excess edges. If $G$ is regular, then using the approach of~\cite{FrKrMa04}, one can obtain a lower bound on the expected number of small tree components in $G_p$ that is sufficiently strong to imply, via~\eqref{eq:excess-comp}, that $\Ex[\exc(G_p)] \ge \eps^3n/12$. Proposition~\ref{prop:excess-conc} will then imply the statement of Corollary~\ref{cor:excess} with $c$ replaced by $\eps^2/20$, which is sufficient for the proof of Theorem~\ref{thm:cycle}.

Moreover, following the approach of~\cite{FrKrMa04}, one may also show that for some absolute constant $C$, with high probability only at most $C \eps n$ vertices of $G_p$ lie in connected components that are not small trees. Consequently, in the proof of Theorem~\ref{thm:cycle}, we may find a rooted tree $R$ such that the average degree of $G'[V(R)]$ is at least $\frac{\eps^2}{10C} \cdot k$. This is enough to show that with high probability, $G_p$ contains a cycle of length at least $\frac{\eps^2}{20C} \cdot k$.

\subsection{Graphs with large average degree}
\label{sec:large-avg-deg}

It is natural to ask whether Theorem~\ref{thm:cycle} is still true
when one replaces the assumption that the minimum degree of $G$ is
at least $k$ with the much weaker assumption that the {\em average}
degree of $G$ is at least $k$. Since each graph $G$ with average
degree $k$ contains a subgraph with minimum degree exceeding $k/2$,
one can easily deduce such statement if one strengthens the
assumption on $p$ to $p \ge \frac{2+\eps}{k}$. The following easy
argument shows that this is not optimal.

Assume that $G$ has $n$ vertices and average degree $k$. For every
$p \in [0,1]$, since the function $(0,\infty) \ni x \mapsto (1-p)^x$
is convex, the expected number of isolated vertices in $G_p$ is at
least $(1-p)^k \cdot n$. Let $c_0$ be the positive solution of the
equation $\frac{c}{2} - 1 + e^{-c} = 0$ and observe that $c_0
\approx 1.6$. Using~\eqref{eq:excess-comp}, we see that if $c > c_0$
and $k$ is sufficiently large, then $\Ex[\exc(G_p)] = \Omega(n)$. It
follows that the assumption on $p$ can be weakened to $p \ge
\frac{c_0 + \eps}{k}$. We believe that similarly as in the minimum
degree case, the assumption $p \ge \frac{1+\eps}{k}$ is sufficient,
but at the moment we are unable to establish this claim.

\medskip

\noindent
\textbf{Acknowledgment.} The authors would like to thank Ron Peled for helpful discussions and the anonymous referee for their invaluable feedback.

\bibliographystyle{amsplain}
\bibliography{min-deg-H-free}

\end{document}